\documentclass[10pt,reqno]{amsart}
\usepackage[utf8]{inputenc}  
\usepackage[T1]{fontenc}
\usepackage{graphicx}
\usepackage{cite}
\usepackage{fixltx2e}
\usepackage{newunicodechar}

\newtheorem{thm}{Theorem}[section]

\newtheorem{prop}[thm]{Proposition}

\newtheorem{rem}[thm]{Remark}

\theoremstyle{definition}
\newtheorem{exmp}{Example}[section]

\numberwithin{equation}{section}

\DeclareMathAlphabet{\mathpzc}{OT1}{pzc}{m}{it}

\newcommand{\N}{\mathbb N}
\newcommand{\R}{\mathbb R}
\newcommand{\C}{\mathbb C}
\newcommand{\K}{\mathbb K}
\newcommand{\bA}{\mathbb A}
\newcommand{\bB}{\mathbb B}
\newcommand{\s}{\mathbb S}
\newcommand{\kL}{\mathcal L}
\newcommand{\ve}{\varepsilon}
\newcommand{\wt}{\widetilde}
\newcommand{\wh}{\widehat}
\newcommand{\ov}{\overline}
\newcommand{\rd}{\mathrm{d}}
\newcommand{\p}{\partial}
\newcommand{\bqn}{\begin{equation}}
\newcommand{\eqn}{\end{equation}}

\newcommand{\PV}{\mathop{\rm PV}\nolimits}

\begin{document}

\title[The principle of linearized stability for  quasilinear evolution equations]{On the principle of linearized stability  in interpolation spaces for  quasilinear evolution equations}

\author{Bogdan--Vasile Matioc}
\address{Fakult\"at f\"ur Mathematik, Universit\"at Regensburg,   93053 Regensburg, Deutschland.}
\email{bogdan.matioc@ur.de}

\author{Christoph Walker}
\address{Leibniz Universit\"at Hannover\\
Institut f\"ur Angewandte Mathematik\\
Welfengarten 1\\
30167 Hannover\\
Germany}
\email{walker@ifam.uni-hannover.de}

\begin{abstract}
 We give a proof for the asymptotic exponential stability of equilibria of quasilinear parabolic evolution equations in admissible interpolation spaces.
\end{abstract}


\maketitle
\section{Introduction}\label{Sec:1}

The principle of linearized stability is a well-known technique in various nonlinear evolution equations for proving stability of equilibria. There is a vast literature on this topic under different assumptions, see e.g. \cite{DaPL88,D89,G88,Lu85,L95,PF81,P02,PSZ09,PSZ09b} though this list is far from being complete. For autonomous fully nonlinear parabolic problems
$$
\dot v=F(v)\,,\quad t\ge 0\,,\qquad v(0)=v^0\,,
$$ with $F\in C^{2-}(E_1,E_0)$\footnote{ We use  $C^{k-}$, $1\leq k\in\N$, to denote the space of functions which possess a locally Lipschitz continuous $(k-1)$th derivative.   
Similarly, $C^\vartheta$ with $\vartheta\in(0,1)$ denotes local H\"older
continuity.}, the use of H\"older maximal regularity allows one to obtain stability of an equilibrium $v_*\in E_1$ in the domain $E_1$ of the Fr\'echet derivative $\p F(v_*)$, see \cite{DaPL88,Lu85,L95,PSZ09,PSZ09b}. 
The stability issue can also be addressed based on maximal $L_p$-regularity of $\p F(v_*)$ in the real interpolation space $(E_0,E_1)_{1-1/p,p}$
or based on continuous maximal regularity of $\p F(v_*)$ in the continuous interpolation space $(E_0,E_1)_{\mu,\infty}^0$, see
\cite{P02,PSZ09,PSZ09b}. These results apply in particular to quasilinear parabolic problems
\begin{equation}\label{EE}
\dot v+A(v)v=f(v)\,,\quad t>0\,,\qquad v(0)=v^0\,.
\end{equation}
For such problems, however, there are other meaningful choices for phase spaces on which the nonlinearities $A$ and $f$ are defined. In \cite{Am93} a complete well-posedness theory for parabolic equations of the form \eqref{EE} is outlined in general
interpolation spaces $E_\alpha=(E_0,E_1)_\alpha$ with an arbitrary admissible interpolation functor $(\cdot,\cdot)_\alpha$ and outside the setting of
 maximal regularity.  
 The main goal of the present research is to
establish the principle of linearized stability for \eqref{EE} in general interpolation spaces $E_\alpha$  within the framework of~\cite{Am93}. \\

We consider in this paper quasilinear evolution equations of the form \eqref{EE}.
Throughout we  let $E_0$ and $E_1$ be Banach spaces over $\K\in \{\R,\C\}$ with continuous and dense embedding 
$$
E_1 \stackrel{d}{\hookrightarrow} E_0\,.
$$
We further let $\mathcal{H}(E_1,E_0)$ denote the open subset of $\kL(E_1,E_0)$  
consisting of negative generators of
strongly continuous analytic semigroups. 
More precisely, $A\in {\mathcal H}(E_1,E_0)$ if
$-A,$ considered as an unbounded operator in $E_0$ with  domain $E_1$, generates a strongly continuous  analytic  semigroup in $\kL(E_0).$

 Given $\theta\in (0,1)$, we fix an admissible interpolation  functor $(\cdot,\cdot)_\theta$, that is,
$$
E_1\stackrel{d}{\hookrightarrow} E_\theta:= (E_0,E_1)_\theta \stackrel{d}{\hookrightarrow} E_0
$$
and put $\|\cdot\|_\theta:=\|\cdot\|_{E_\theta}$.
We further fix
\begin{equation}\label{a1a}
0<\gamma\le \beta<\alpha < 1
\end{equation}
and assume that
\begin{equation}\label{a1b}
\emptyset\not= O_\beta\ \text{ is open in }\ E_\beta\,.
\end{equation}
Then $O_\delta:=O_\beta\cap E_\delta$ is, for $\delta\in[\beta, 1]$,  an open subset of $E_\delta$. The operators $A$ and $f$  in \eqref{EE} are assumed to satisfy
\begin{equation}\label{a1c}
(A,f)\in C^{1-}\big(O_\beta,\mathcal{H}(E_1,E_0)\times E_\gamma\big)\,.
\end{equation}

Given $v^0\in O_\alpha$,  we are interested in the existence and the qualitative properties of classical solutions  to \eqref{EE}, that is, of  functions
\[
v\in C(\dot I,E_1)\cap C^1(\dot I,E_0)\cap C(I,O_\alpha)
\]
which satisfy \eqref{EE} pointwise. Here, the interval  $I$ is either $[0,T)$ or $[0,T]$ for some $T\in(0,\infty]$   and $\dot I:= I\setminus\{0\}$.

Clearly, the well-posedness of quasilinear or even fully nonlinear problems has attracted considerable attention in the past. There are many results on abstract equations, including e.g. \cite{A86,Am88,A90,Am93,A05,G88,Lu84,Lu85,Lu85b,PS16,PSW18} for quasilinear and e.g. \cite{AT88,DaPL88,DaP96,DaPG79,Lu87,L95} for fully nonlinear problems
using different tools from theories on semigroups, evolution operators, and maximal regularity (though none of these lists is complete, of course). In particular, the next theorem is stated in  \cite[Theorem 12.1, Theorem 12.3]{Am93} (see also \cite{A86}). 
 We will use part of its proof  when subsequently establishing the stability result in Theorem~\ref{T2}. For this reason and since a complete proof of this result merely assuming \eqref{a1a}-\eqref{a1c} does not seem to be available in the literature to the best of our knowledge (though it is contained  in \cite{A86} for nonautonomous problems under slightly stronger assumptions), we include it here. 
We also refer to \cite[Appendix B]{M16x} for a proof for the homogeneous case $f=0$ and to \cite[Theorem 1.1]{ELW15} for
 a proof in a concrete application (see also \cite[Theorem 2.2]{ChW10} for a similar fixed point argument).\\

\begin{thm}\label{T1}
Suppose \eqref{a1a}-\eqref{a1c}. 
\begin{itemize}
\item[(i)] {\em (Existence)} Given any $v^0\in O_\alpha$, the Cauchy problem \eqref{EE} possesses a maximal classical solution
\begin{equation*} 
 v=v(\cdot;v^0)\in C^1((0,t^+(v^0)),E_0)\cap C((0,t^+(v^0)),E_1)\cap  C ([0,t^+(v^0)),O_\alpha)\,,
\end{equation*}
where $ t^+(v^0)\in(0,\infty]$,  such that 
\begin{equation*}
  v(\cdot;v^0)\in C^{\alpha-\eta}([0,t^+(v^0)),E_\eta)\,,\quad \eta\in [0,\alpha]\,.
\end{equation*}
  If $v^0\in O_1$, then $v(\cdot;v^0)\in C^1([0,t^+(v^0)),E_0)\cap C([0,t^+(v^0)),E_1)$.
\item[(ii)]  {\em (Uniqueness)} If $$\wt v\in C((0,T],E_1)\cap C^1((0,T],E_0)\cap C^\vartheta([0,T],O_\beta)$$ solves
\eqref{EE} pointwise for some $T>0$ and $\vartheta\in (0,1)$, then   $\wt v=v(\cdot;v^0)$  on $[0,T]$.
\item[(iii)] {\em (Continuous dependence)} The mapping $[(t,v^0)\mapsto v(t;v^0)]$ defines a semiflow\footnote{That is, $\mathcal{D}:=\{(t,v^0)\,:\, 0\le t<t^+(v^0)\,,\,v^0\in O_\alpha\}$ is open in $[0,\infty)\times E_\alpha$ and the function $$v:\mathcal{D}\rightarrow E_\alpha\,,\quad (t,v^0)\mapsto v(t;v^0)$$ is continuous with $v(0;v^0)=v^0$ and $v(t;v(s;v^0))=v(t+s;v^0)$ for $0\le s< t^+(v^0)$ and $0\le t<t^+(v(s;v^0))$.} on $O_\alpha$.
\item[(iv)] {\em (Global existence)} If the orbit $v([0,t^+(v^0));v^0)$ is relatively compact in $O_\alpha$, then $t^+(v^0)=\infty$.
\item[(v)] {\em (Blow-up criterion)} Let $v^0\in O_\alpha$ be such that $t^+(v^0)<\infty$. Then:
\begin{itemize}
\item[($\alpha$)] If
$v(\cdot;v^0):[0,t^+(v^0))\to E_\alpha$ is uniformly continuous, then
\begin{equation}\label{1a} 
 \lim_{t\nearrow t^+(v^0)}\mathrm{dist}_{E_\alpha}\big( v(t;v^0),\partial O_\alpha\big)=0\,.
\end{equation}
\item[($\beta$)]  If $E_1$ is compactly embedded in $E_0$, then 
\begin{equation}\label{1}
\lim_{t\nearrow t^+(v^0)}\|v(t;v^0)\|_{\eta}=\infty\qquad\text{ or }\qquad \lim_{t\nearrow t^+(v^0)}\mathrm{dist}_{E_\beta}\big( v(t;v^0),\partial O_\beta\big)=0
\end{equation}
for each $\eta\in (\beta,1)$.
\end{itemize}
\end{itemize}
\end{thm}

The proof of Theorem~\ref{T1} is given in Section~\ref{Sec:2}.

\begin{rem}\label{R2}
\begin{itemize}
\item[(i)]   Theorem~\ref{T1} remains true if the assumption $f\in C^{1-}(O_\beta,E_\gamma)$ in \eqref{a1c} is replaced by the assumption that $f:O_\beta\rightarrow E_\gamma$ is bounded on bounded sets
  and that $f\in C^{1-}(O_\beta,E_0)$, see Remark~\ref{R1}.
\item[(ii)]    Actually, the maximal solution from  Theorem~\ref{T1} is 
unique among all solutions in $C^\vartheta([0,T],O_\beta)$  to the fixed point formulation \eqref{ttp} (see the proof of Proposition~\ref{P1}). 
In some applications, see e.g. \cite[Theorem 1.1]{M16x} or \cite[Theorem 1.1]{M17x}, one can establish, by exploiting the quasilinear structure of the equation, a priori  H\"older regularity with respect to time in $E_\beta$ (for some 
small H\"older exponent $\vartheta$) for classical solutions. This allows one to apply the uniqueness result of Theorem~\ref{T1}.  
A further advantage of the uniqueness feature of Theorem~\ref{T1} is that one can use parameter tricks to improve the space-time regularity of the solutions, cf. e.g. \cite[Theorem 1.3]{M16x}.
\item[(iii)] As stated in \cite{Am93}, the mapping 
$$[(t,v^0)\mapsto v(t;v^0)]:\{(t,v^0)\,:\, 0<t<t^+(v^0)\,, v^0\in O_\alpha\}{\rightarrow E_\alpha}$$ 
belongs to the class $C^k$, 
$k\in \N\cup\{\infty,\omega\}$ (where $\omega $ denotes real-analyticity) if $$(A,f)\in C^k\big(O_\beta,\mathcal{H}(E_1,E_0)\times E_\gamma\big)\,.$$ 
The main ideas of proof for this claim can be found in \cite[Theorem 11.3]{Am88} (see also \cite{Lu82}).
\end{itemize}
\end{rem}

 Next, we turn to the stability of equilibria to quasilinear parabolic problems 
in interpolation spaces. 
In order to present the general result  we suppose  that 
\begin{equation}\label{a1e}
v_*\in O_1 \quad\text{with}\quad A(v_*)v_*=f(v_*)
\end{equation}
is an equilibrium solution to \eqref{EE}.
Additionally, we assume that 
\begin{equation}\label{a1f}
f:O_\beta\rightarrow E_0 \quad\text{and}\quad A(\cdot)v_*:O_\beta\rightarrow  E_0 \quad\text{are Fr\'echet differentiable at}\ v_* 
\end{equation}
with Fr\'echet derivatives $\p f(v_*)$, respectively, and $(\p A(v_*)[\cdot])v_*$,
and that the linearized operator
 $$\bA:=  -A(v_*)-(\p A(v_*)[\cdot])v_*+\p f(v_*)$$
  has a negative spectral bound, that is,
\begin{equation}\label{a1g}
-\omega_0:=\sup\left\{\mathrm{Re}\, \lambda\,:\, \lambda \in \sigma(\bA)\right\}<0\,.
\end{equation}


The following result now states that, under the assumptions above,
the equilibrium $v_*$ is asymptotically exponentially stable 
in the phase space~$E_\alpha=(E_0,E_1)_\alpha$ with an arbitrary admissible interpolation functor $(\cdot,\cdot)_\alpha$ (provided \eqref{a1a}-\eqref{a1c} hold), the proof is contained in Section~\ref{Sec:3}. 
As mentioned before, stability results in $E_1$ can be found e.g. in \cite{DaPL88,Lu85,L95,PSZ09,PSZ09b}, and stability results in the  real interpolation space $(E_0,E_1)_{1-1/p,p}$
or in the continuous interpolation space $(E_0,E_1)_{\mu,\infty}^0$ can be found e.g. in
\cite{P02,PSZ09,PSZ09b}. 
In Theorem \ref{T2}, however, we may choose e.g. $(\cdot,\cdot)_\alpha$ to be the complex interpolation functor, which is not possible for the other stability results, and, moreover, we require less assumptions as we do not assume  that the Fr\'echet derivative $\bA$ has maximal regularity.

\begin{thm}[Exponential stability]\label{T2}
Suppose in addition to \eqref{a1a}-\eqref{a1c} that also \eqref{a1e}-\eqref{a1g} hold. Then $v_*$ is asymptotically exponentially stable
 in $E_\alpha$. More precisely, given any $ \omega\in (0,\omega_0)$, there is $\ve_0>0$ and $M\ge 1$ such that, 
 for each $v^0\in \ov\bB_{E_\alpha}(v_*,\ve_0)$, the unique solution to \eqref{EE} exists globally in time and
\begin{equation}\label{stable}
\|v(t;v^0)-v_*\|_\alpha\le M e^{-\omega t}\|v^0-v_*\|_\alpha\,,\quad t\ge 0\,.
\end{equation}
\end{thm}

We also prove instability of an equilibrium $v_*\in O_1$ assuming that
\begin{equation}\label{a2a}
(A,f)\in C^{2-}\big(O_1,\kL(E_1,E_0)\times  E_0\big)
\end{equation}
with
\begin{equation}\label{a2b}
 \p f(v_*),\, (\p A(v_*)[\cdot])v_* \in \kL(E_{\eta},E_0)\quad\text{for some $\eta\in[\beta,1)$}\,.
\end{equation}
The latter condition is in accordance with \eqref{a1c}.
With respect to the spectrum of the linearized operator $\bA$ from above we now require that  
\begin{equation}\label{a2c}
 \left\{
 \begin{array}{lll}
 \sigma_+(\bA):=\{\lambda\in \sigma(\bA)\,:\,  {\rm Re\,}\lambda>0\}\neq \emptyset,\\[1ex]
 \omega_+:=\inf \{{\rm Re\,}\lambda\,:\, \lambda\in \sigma_+(\bA)\}>0\,.
 \end{array}
 \right. 
\end{equation}
These conditions guarantee that $v_*$ is unstable  in the phase space $E_\alpha$ 
as shown in the next theorem. 
We point out that this result is closely related to \cite[Theorem 9.1.3]{L95}, 
where instability in $E_1$ is proven, and follows along the lines of the proof of the latter.

\begin{thm}[Instability]\label{T3} Let $v_*\in O_1$ be an equilibrium to \eqref{EE}.
Suppose in addition to \eqref{a1a}-\eqref{a1c} that also \eqref{a2a}-\eqref{a2c} hold. Then $v_*$ is unstable
 in $E_\alpha$. More precisely, there exists a neighborhood $U$ of $v_*$ in $O_\alpha$ such that for each $n\in\N^*$ 
 there are initial data
 $v_{n}^0\in \bB_{E_\alpha}(v_*,1/n)\cap O_\alpha$ for which the corresponding solution $v(\cdot; v_n^0)$ satisfies
 \[
 v(t; v_n^0)\not\in U\quad\text{for some $t\in(0,t^+(v_n^0))$.}
 \]
\end{thm}

The next section contains the proof of Theorem~\ref{T1} while the proofs of Theorem~\ref{T2} and Theorem~\ref{T3} are given in Section~\ref{Sec:3}. In Section~\ref{Sec:4} we present applications of the stability result to the Muskat problem with and without surface tension.

\section{Proof of Theorem \ref{T1}}\label{Sec:2}

We present here only the proof in the case when $\K=\C$, the arguments 
in the case $\K=\R$ being identical (but require some additional clarification).

For $\omega>0$ and $\kappa\ge 1$,  let
$\mathcal{H}(E_1,E_0,\kappa,\omega)$ be the class of all $A\in\mathcal{L}(E_1,E_0)$ such that $\omega+A$ is an
isomorphism from $E_1$ onto $E_0$ and satisfies the resolvent estimates
$$
\frac{1}{\kappa}\,\le\,\frac{\|(\mu+A)z\|_{0}}{\vert\mu\vert \,\| z\|_{0}+\|z\|_{1}}\,\le \, \kappa\ ,\qquad {\rm Re\,}
\mu\ge \omega\ ,\quad z\in E_1\setminus\{0\}\ .
$$
Then $A\in\mathcal{H}(E_1,E_0,\kappa,\omega)$ implies that $A\in\mathcal{H}(E_1,E_0)$, see \cite[I.Theorem 1.2.2]{Am95}.
 On the other hand, $A\in\mathcal{H}(E_1,E_0)$ implies that there are $\omega>0$ and $\kappa\ge 1$ 
 such that $A\in\mathcal{H}(E_1,E_0,\kappa,\omega)$.\\

Let us recall that for each $A\in C^\rho(I,\mathcal{H}(E_1,E_0))$, with $\rho\in(0,1)$, there exists a unique parabolic evolution operator $U_A(t,s)$, $0\le s\le t<\sup I$, in the sense of \cite[II. Section 2]{Am95}. This enables one to reformulate the Cauchy problem \eqref{EE} as a fixed point equation of the form
\begin{equation}\label{ttp}
v(t)=U_{A(v)}(t,0)v^0+\int_0^t U_{A(v)}(t,s) f(v(s))\,\rd s\,,\quad t\in I\,,
\end{equation}
in the class of functions $v\in C^\rho(I,E_\beta)$, see \cite[II.Theorem 1.2.1]{Am95}.  
The linear theory outlined in \cite[Chapter~II]{Am95} (in particular, see \cite[II. Section 5]{Am95}) plays an important  role in the subsequent analysis.\\

The following uniform local existence and uniqueness result is  fundamental for the proof of Theorem~\ref{T1}.

\begin{prop}\label{P1}
Let \eqref{a1a}-\eqref{a1c} hold true and let $S_\alpha \subset O_\alpha$ be any compact subset of $E_\alpha$.
Then, there are a neighborhood $U_\alpha$ of $S_\alpha$ in $O_\alpha$ and $T:=T(S_\alpha)>0$ such that, for each $v^0\in U_\alpha$,
the problem \eqref{EE}
has a classical solution 
\[
v=v(\cdot;v^0)\in C^1((0,T],E_0)\cap C((0,T],E_1)\cap C ([0,T],O_\alpha) \cap C^{\alpha-\eta}([0,T],E_\eta)\,, \quad \eta\in[0,\alpha].
\]
Moreover, there is a constant $c_0:=c_0(S_\alpha)>0$ such that
$$
\| v(t;v^0)-v(t;v^1)\|_\alpha\le c_0 \| v^0-v^1\|_\alpha\,,\qquad 0\le t\le T\,,\quad v^0, v^1\in U_\alpha\,.
$$
Finally, if $\vartheta\in(0,1)$ and $\wt v\in C^1((0,T],E_0)\cap C((0,T],E_1)\cap C^\vartheta ([0,T],O_\beta)$ with $\wt v (0)=v^0 \in U_\alpha$
solves  \eqref{EE} pointwise, then $\wt v=v(\cdot;v^0)$.
\end{prop}

\begin{proof}  
{  Let $S_\alpha \subset O_\alpha$ be a compact subset of $E_\alpha$.
Since   $S_\alpha \subset O_\beta$  is also compact in $E_\beta$, there is 
$\delta>0$ such that  $\mathrm{dist}_{E_\beta}(S_\alpha, \partial O_\beta) > 2\delta > 0$.
Furthermore, 
$A$ as well as $f$ are uniformly Lipschitz continuous on some neighborhood 
of $S_\alpha$ (see \cite[II.Proposition 6.4]{A83}), 
that is, there are $\ve > 0$ and $L>0$ such that  $\ov\bB_{E_\beta}(S_\alpha, 2\ve) 
\subset \mathbb{B}_{E_\beta}(S_\alpha, \delta) \subset O_\beta$\footnote{Given $\theta\in(0,1)$, $S\subset E_\theta$, and $\delta>0$, 
we denote by $\mathbb{B}_{E_\theta}(S, \delta)$
the set $\{x\in E_\theta\,:\, \mathrm{dist}_{E_\theta}\big(x,S\big)<\delta\}$.} and
\bqn\label{est}
\|A(x) - A(y)\|_{\mathcal{L}(E_1,E_0)}+ \|f(x) - f(y)\|_\gamma \le L\|x-y\|_\beta\ , \quad x,y \in 
\ov\bB_{E_\beta}(S_\alpha, 2\ve)\ . 
\eqn
Moreover, since $A(S_\alpha)$ is compact in $\mathcal{H}(E_1,E_0)$, it follows 
from \cite[I.Corollary 1.3.2]{Am95} that there are $\kappa\ge 1$ and $\omega>0$ such that we may assume
without loss of generality (by making $\ve>0$ smaller, if necessary) that 
\bqn\label{est1}
A(x)\in \mathcal{H}(E_1,E_0,\kappa,\omega)\,,\quad x \in \ov\bB_{E_\beta}(S_\alpha, 2\ve)\,.
\eqn
Also note from \eqref{est} that there is $b > 0$ with
\bqn\label{est3}
\|f(x)\|_0+\|f(x)\|_\gamma \le b\ , \quad x \in \ov\bB_{E_\beta}(S_\alpha, 2\ve)\, . 
\eqn
Fix $\rho\in (0,\alpha-\beta)$. Given $T \in (0,1)$ we introduce 
$$\mathcal{V}_T := \big\{v \in C([0,T], \ov\bB_{E_\beta}(S_\alpha, 2\ve)) \,:\, \|v(t) - v(s)\|_\beta 
\le |t-s|^{\rho}\ , ~0 \le s,t \le T\big\}$$
and observe that this set is closed in $C([0,T], E_\beta)$, hence complete. Thus, if $v\in\mathcal{V}_T$, then 
\bqn\label{lp}
A(v(t))\in \mathcal{H}(E_1,E_0,\kappa,\omega)\,,\quad t\in [0,T]\,,
\eqn
and
\bqn\label{lq}
A(v)\in C^\rho\big([0,T],\mathcal{L}(E_1,E_0)\big)\quad\text{with}\quad \sup_{0\le s<t\le T}
\frac{\|  A(v(t))-  A(v(s))\|_{\mathcal{L}(E_1,E_0)}}{( t-s)^\rho}\le L\,.
\eqn 
In particular, for each $v\in\mathcal{V}_T$, the evolution operator
$$
U_{A(v)}(t,s)\,,\quad 0\le s\le t\le T\,,
$$
generated by $ A(v)\in C^\rho\big([0,T],\mathcal{H}(E_1,E_0)\big)$ is well-defined, and \eqref{lp}-\eqref{lq} guarantee that we are 
in a position to use the results of \cite[II.Section 5]{Am95}.

Let $c_{\alpha,\beta}$ be the norm of the embedding $E_\alpha\hookrightarrow E_\beta$. 
 Then
$U_\alpha:=\mathbb{B}_{E_\alpha}(S_\alpha,\ve/(1+c_{\alpha,\beta}))\subset O_\alpha.$
Given $v^0\in U_\alpha,$ define
\begin{equation}\label{Lambda}
\Lambda(v)(t):= U_{A(v)}(t,0)v^0 +\int_0^t U_{A(v)}(t,s) f(v(s))\,\rd s\,,\quad t\in [0,T]\,,\quad v\in\mathcal{V}_T\,.
\end{equation}
Then, according to \cite[II.Theorem 5.3.1]{Am95}, there are   constants $c>0$ and $\nu\ge 0$ (possibly depending on $\kappa$, $\omega$, $L$,
and $\rho$, but not on $T$) such that, for $v\in\mathcal{V}_T$ and $0\le s\le t\le  T\le 1$, we have, using \eqref{est3},
\begin{equation}\label{lip}
\begin{aligned}
\|\Lambda(v)(t)-\Lambda(v)(s)\|_\beta &\le c e^{\nu t}\big( \|v^0\|_\alpha +\|f(v)\|_{L_\infty((0,t),E_0)}\big) (t-s)^{\alpha-\beta} \\
&\le c e^{\nu}T^{\alpha-\beta-\rho} \big( \|v^0\|_\alpha +b\big) (t-s)^\rho  
\end{aligned}
\end{equation}
and, in particular with $\Lambda(v)(0)=v^0$,
\begin{equation*}
\begin{split}
\|\Lambda(v)(t)-v^0\|_\beta &\le c e^{\nu} \big( \|v^0\|_\alpha +b\big) T^{\alpha-\beta}\,.
\end{split}
\end{equation*}
Furthermore, \cite[II.Theorem 5.2.1]{Am95} and \eqref{est} imply,
for $v,w\in\mathcal{V}_T$ and $0\le s\le t\le T\le 1$, that
\begin{align}
\|\Lambda(v)(t)&-\Lambda(w)(t)\|_\beta \nonumber\\
 &\le c e^{\nu t}\left\{T^{\alpha-\beta} \| A(v)-A(w)\|_{C([0,T],\mathcal{L}(E_1,E_0)}\left(\|v^0\|_\alpha + T^{1-\alpha+\gamma}
 \|f(v)\|_{C([0,T],E_\gamma)}\right)\right. \nonumber\\
&\hspace{3cm} \left. +T^{1-\beta} \|f(v)-f(w)\|_{C([0,T],E_0)}\right\} \nonumber\\
&\le c_1 e^{\nu} \big(1+ \|v^0\|_\alpha + b\big) T^{\alpha-\beta} \|v-w\|_{C([0,T],E_\beta)}\,.\label{po}
\end{align}
Consequently, there is $T:=T(S_\alpha)\in (0,1)$ such that, for $v^0\in U_\alpha $, the mapping
\mbox{$\Lambda:\mathcal{V}_T\rightarrow \mathcal{V}_T$} is a contraction and thus possesses a unique fixed point 
$v(\cdot;v^0)\in \mathcal{V}_T$. The linear theory in \cite[II.Theorems 1.2.1  and 5.3.1]{Am95} implies that $v(\cdot;v^0)$ 
is a solution to \eqref{EE} with regularity properties as claimed. This proves the existence statement. 

Next, to prove Lipschitz continuity with respect to the initial values, let $v^0, v^1\in U_\alpha$.
Then we derive from  \cite[II.Theorem 5.2.1]{Am95} for $\mu\in \{\beta,\alpha\}$ and $t\in [0,T]$, by using \eqref{est} and \eqref{est3}, that
\begin{align}
&\|v(t;v^0)-v(t;v^1)\|_\mu \nonumber\\
&\leq ce^{\nu T}\left\{T^{\alpha-\mu}\|A(v(\cdot; v^0))-A(v(\cdot; v^1))\|_{C([0,T],\kL(E_1,E_0))}\left[\|v_0\|_\alpha
+T^{1-\alpha+\gamma}\|f(v(\cdot;v^0))\|_{L_\infty((0,t),E_\gamma)}\right]\right.\nonumber\\
&\left. \qquad\qquad +\|v^0-v^1\|_\mu+T^{1-\mu}\|f(v(\cdot;v^0))-f(v(\cdot;v^1))\|_{L_\infty((0,T),E_0)}\right\}\nonumber\\
&\leq c\big(\|v^0-v^1\|_\mu+ T^{\alpha-\mu}\|v(\cdot;v^0)-v(\cdot;v^1)\|_{C([0,T],E_\beta)}\big)\,.\label{p}
\end{align}
Taking first $\mu=\beta$ in \eqref{p} and making $T>0$ smaller, if necessary, we obtain
$$
\|v(\cdot;v^0)-v(\cdot;v^1)\|_{C([0,T],E_\beta)}\le c  \|v^0-v^1\|_\beta
$$
and, using this and then \eqref{p} with $\mu=\alpha$, we deduce that indeed
$$
\|v(\cdot;v^0)-v(\cdot;v^1)\|_{C([0,T],E_\alpha)}\leq c_0 \|v^0-v^1\|_\alpha \,,\quad v^0, v^1\in U_\alpha\,,
$$
for some constant $c_0=c_0(S_\alpha)>0$.
}

To prove uniqueness, let $\wt v\in C^\vartheta([0,T],O_\beta)$ solve \eqref{EE}  with $\wt v(0)=v^0\in U_\alpha$.
Making $\rho\in (0,\alpha-\beta)$ smaller than $\vartheta$ and choosing $T$ smaller, if necessary, 
we may assume that $\wt v\in \mathcal{V}_T$. As a solution to \eqref{EE}, $\wt v$ is a fixed point of $\Lambda$, 
hence  $\wt v=v(\cdot;v^0)$ by what was shown above.

\end{proof}

\begin{rem}\label{R1}
 Suppose that the assumption $f\in C^{1-}(O_\beta,E_\gamma)$ in \eqref{a1c} is replaced by the assumption that the function \mbox{$f:O_\beta\rightarrow E_\gamma$} is bounded on bounded sets
  and that $f\in C^{1-}(O_\beta,E_0)$. Then $$\|f(x) - f(y)\|_0 \le L\|x-y\|_\beta\,,\quad x,y \in \ov{\mathbb{B}}_{E_\beta}(S_\alpha, 2\ve)\,,$$ in \eqref{est} while \eqref{est3} still holds. This suffices for \eqref{po} while the rest of the proof of Theorem~\ref{T1} remains the same.
\end{rem} 
 
We now provide the proof of Theorem \ref{T1}.

\begin{proof}[Proof of Theorem \ref {T1}]
 Suppose that \eqref{a1a}-\eqref{a1c} hold. We divide the proof in four parts.\\

 \noindent {\em Existence and Uniqueness:} 
Given $v^0\in O_\alpha=O_\beta\cap E_\alpha$ it readily follows from Proposition~\ref{P1} that the problem \eqref{EE}
admits a unique local solution which, by the uniqueness assertion of Proposition~\ref{P1}, can be extended to 
a maximal solution $v(\cdot;v^0)$ on the maximal interval of existence $[0,t^+(v^0))$ with regularity properties as stated in \eqref{a1b}. 
This proves part (i) and (ii) from Theorem~\ref{T1}.\\

 

 \noindent{\em Continuous dependence:} Let $v^0\in O_\alpha$ and $t_0\in (0,t^+(v^0))$ be arbitrary.
  Then $S_\alpha:=v([0,t_0];v^0)\subset O_\alpha$ is compact. Hence, according to Proposition~\ref{P1},
 there are $\ve>0$, $T>0$, and $c_0\ge 1$  such that $T<t^+(v^1)$ for any
 $v^1\in U_\alpha=\mathbb{B}_{E_\alpha}(S_\alpha,\ve/(1+c_{\alpha,\beta}))$ and
  \begin{align}\label{FE2}
\|v(t;v^1)-v(t; v^2)\|_{\alpha}\leq c_0 \|v^1- v^2\|_\alpha\,,\quad 0\le t\le T\,,\quad v^1, v^2\in U_\alpha\,.
\end{align}
Choose $N\ge 1$ with $(N-1)T<t_0\le NT$ and set $V_\alpha:=\mathbb{B}_{E_\alpha}(v^0,\ve_0)$ for 
  $\ve_0:=\ve/((1+c_{\alpha,\beta})c_0^{N-1})$, 
which is a neighborhood of $v^0$ contained in $U_\alpha$. We now claim that
there exists $k_0\geq1$ such that
 \begin{itemize}
  \item[(a)] $t_0<t^+(v^1)$ for each $v^1\in V_\alpha$,
  \item[(b)] $\|v(t;v^1)-v(t;v^0)\|_{E_\alpha} \leq k_0 \|v^1- v^0\|_\alpha$ for $0\le t\le t_0$ and $v^1\in V_\alpha$.
 \end{itemize}
Indeed, let $v^1\in V_\alpha$. If $t_0\le T$, this is exactly what is stated above. 
If otherwise $T<t_0$ we have $v(T;v^0)\in S_\alpha$ so that the estimate
$$
\|v(t;v^1)-v(t; v^0)\|_{\alpha}\leq c_0 \|v^1- v^0\|_\alpha <\ve/(1+c_{\alpha,\beta})\,,\quad 0\leq t\leq T\,,
$$
implied by \eqref{FE2} yields $v(T;v^1)\in U_\alpha$. Hence $T<t^+(v(T;v^i))$ for $i=0,1$ while uniqueness
of solutions to \eqref{EE} entails that $v(t;v(T;v^i))=v(t+T;v^i)$, $0\le t\le T$.
Thus, \eqref{FE2} shows that
$$
\|v(t+T;v^1)-v(t+T; v^0)\|_{\alpha}\leq c_0^2 \|v^1- v^0\|_\alpha \,,\quad 0\le t\le T\,.
$$
If  $N=2$ we are done. Otherwise we proceed to deduce (a) and (b) after finitely many iterations. From properties (a) and (b) it is immediate that the solution map defines a semiflow in $O_\alpha$.\\

 

 \noindent{\em Global existence:} 
Suppose that $t^+(v^0)<\infty$ and that  $S_\alpha:=cl_{E_\alpha}v([0,t^+(v^0));v^0)\subset O_\alpha$ is compact.
Then Proposition~\ref{P1} ensures the existence of $T>0$ in dependence of $S_\alpha$ such that any solution with initial value in
$S_\alpha$ exists at least on $[0,T]$. Choosing $v(t_0;v^0)$ as initial value with $ t^+(v^0)-t_0 <T$ yields a contradiction.
This proves part (iv) of Theorem~\ref{T1}.\\

 \noindent{\em Blow-up criterion:} Consider $v^0\in O_\alpha$ with $t^+(v^0)<\infty$.

To prove ($\alpha$) from part (v) of Theorem~\ref{T1} let $v(\cdot;v^0):[0,t^+(v^0))\to E_\alpha$ be uniformly continuous and assume that \eqref{1a} is not true. Then $\lim_{t\nearrow t^+(v^0)}v(t;v^0)$ exists in $O_\alpha$, hence $v([0,t^+(v^0));v^0)$ is relatively compact in $O_\alpha$. This contradicts (iv) of Theorem~\ref{T1}.

To prove ($\beta$) from part (v) of Theorem~\ref{T1} let $E_1$ be compactly embedded in $E_0$ and assume that \eqref{1} is not true. 
Since we only assumed that $\beta<\alpha$ in the existence argument, we may assume that $\eta>\alpha$. Since then $E_{\eta}$
embeds compactly in $E_\alpha$, it  follows that $v([0,t^+(v^0));v^0)$ is relatively compact in $O_\alpha$. This again contradicts (iv) of Theorem~\ref{T1}.

 
\end{proof}

\section{Proof of Theorem~\ref{T2} and Theorem~\ref{T3}}\label{Sec:3}

We first establish the exponential stability result stated in Theorem \ref{T2}.

\begin{proof}[Proof of Theorem \ref{T2}]
 Suppose that \eqref{a1a}-\eqref{a1c} and \eqref{a1e}-\eqref{a1g} hold.
 Given   $0<\xi<\zeta<1$, let $c_{\zeta,\xi} $ denote the norm of the continuous embedding  
$E_\zeta\hookrightarrow E_\xi$. Let further $k_0\geq1$ denote the constant from $(b)$
in the   the proof of the continuous dependence claim  of
Theorem \ref{T1} (for $v^0=v_*$ and $t_0=1$ there). We now fix $\eta\in (\beta,\alpha)$ 
and define $c_1:=(1+c_{\alpha,\beta})(1+c_{\eta,\beta})(1+c_{\alpha,\eta}) k_0$.
 Since $O_\beta$ is an open subset of $E_\beta,$ there exists $\ve>0$ such that $\ov\bB_{E_\alpha}(v_*,3\ve/c_1\big)
 \subset \ov\bB_{E_\beta}(v_*,3\ve\big)\subset O_\beta$. 
 Thus, given any $u^0\in \ov\bB_{E_\alpha}(0,3\ve/c_1\big)$ the evolution problem \eqref{EE}
 with $v^0:=u^0+v_*$ has a unique maximal classical solution $v(\cdot;v^0)$  
  on $[0,t^+(v^0))$  with properties as stated in Theorem~\ref{T1}. This implies in particular that 
\begin{equation}\label{3}
u:=v(\cdot;v^0)-v_* \in C^{\alpha-\eta}([0,t^+(v^0)),E_\eta)\,,\quad  \eta\in [0,\alpha]\,,
\end{equation}
is a classical solution to the equation
\begin{equation}\label{4}
\dot u+\wh A(u)u=\wh f(u)\,,\quad t>0\,,\qquad u(0)=u^0\,,
\end{equation}
where we defined, for $w\in \wh O_\beta:=O_\beta-v_*$,
$$
\wh A(w):=A(w+v_*)-\p f(v_*) + (\p A(v_*)[\cdot])v_*
$$
and
$$
\wh f(w):= f(w+v_*)-A(w+v_*)v_*-\p f(v_*)w + (\p A(v_*)[w])v_*\,.
$$
 Observe that 
\begin{equation}\label{5}
\wh A\in C^{1-}\big(\wh O_\beta,\mathcal{H}(E_1,E_0)\big)\,,
\end{equation}
since $(\p A(v_*)[\cdot])v_*-\p f(v_*)\in\mathcal{L}(E_\beta,E_0)$ may be considered as perturbation due to \cite[I.Theorem 1.3.1]{Am95}.
Also note that, given any $\xi\in(0,1)$, we may assume due to \eqref{a1e} (after making $\ve>0$ smaller, if necessary) that
\begin{equation}\label{10b}
\|\wh f(w)\|_0\le\xi \|w\|_\beta\,,\quad w\in \ov\bB_{E_\beta}(0,2\ve)\,.
\end{equation}
We now show that zero is an exponentially asymptotically stable equilibrium to \eqref{4} in the $E_\alpha$-topology.
For this, let $ \omega\in (0,\omega_0)$ be arbitrary and set $4\delta:=\omega_0-\omega>0$ and 
$\rho:=\alpha-\eta>0$. 
We then obtain from \cite[I.Proposition 1.4.2]{Am95} and \eqref{5} 
that there are $\kappa\ge 1$ and $L>0$ such that (making again $\ve>0$ smaller, if necessary)
\begin{equation}\label{7a}
-\omega_0+\delta + \wh A(w)\in \mathcal{H}(E_1,E_0,\kappa,\delta)\,,\quad w\in\ov\bB_{E_\beta}(0,2\ve)\,,
\end{equation}
and
\begin{equation}\label{7b}
\|\wh A(w_1)-\wh A(w_2)\|_{\mathcal{L}(E_1,E_0)}\le L \| w_1-w_2\|_\beta\,,\quad w_1,w_2\in \ov\bB_{E_\beta}(0,2\ve)\,.
\end{equation}
We then denote by $c_0(\rho)>0$ the constant from \cite[II.Theorem 5.1.1]{Am95} and choose $N>0$ such that $c_0(\rho)N^{1/\rho}=\delta$. 
Given $T\in(0,\infty)$, we introduce
$$
\mathcal{M}(T):=\left\{ w\in C\big([0, T],\ov\bB_{ E_\eta}(0,2\ve/c_{\eta,\beta})\big)\,:\, \|w(t)-w(s)\|_\eta
\le \frac{N}{L c_{\eta,\beta}}\vert t-s\vert^\rho\,,\, 0\le s,t\le T\right\}\,.
$$
Note that $\ov\bB_{E_\eta}(0,2\ve/ c_{\eta,\beta})\subset \ov\bB_{E_\beta}(0, 2\ve)$. 
We then derive from \eqref{7a}-\eqref{7b} for $w\in\mathcal{M}(T)$ that
\bqn\label{b1}
-\omega_0+\delta + \wh A(w(t))\in \mathcal{H}(E_1,E_0,\kappa,\delta)\,,\quad t\in [0,T]\,,
\eqn
and
\bqn\label{b2}
\wh A(w)\in C^\rho\big([0,T],\mathcal{L}(E_1,E_0)\big)
\quad\text{with}\quad \sup_{0\le s<t\le T} \frac{\|\wh A(w(t))-\wh A(w(s))\|_{\mathcal{L}(E_1,E_0)}}{(t-s)^\rho}\le N\,.
\eqn
Owing to the statements (a) and (b) in the proof of   
Theorem \ref{T1} (with $v^0=v_*$ and $t_0=1$ there)  we may assume,
again after making $\ve>0$ smaller, that    $t^+(v^0)\geq 1$ and
$\|u(t)\|_\eta\leq  \ve/c_{\eta,\beta}$ for all $t\in[0,1]$ and all  $u^0\in\ov\bB_{E_\alpha}(0,\ve/c_{1}\big)$. 
Choosing $S_\alpha:=\{v_*\}$ in Proposition \ref{P1} and making  $\ve$ smaller such that
$\ov\bB_{E_\alpha}(v_*,\ve/c_{1}\big) \subset U_\alpha$,
it follows then from the proof of Proposition \ref{P1} (see \eqref{lip}) that there exist $0<t_0\leq 1$  and $c_2>0$ such that 
\[
\|u(t)-u(s)\|_\beta=\|v(t;v^0)-v(s;v^0)\|_\beta\leq c_2(t-s)^{\rho} ,\quad
0\leq s\leq t\leq t_0\,,\quad u^0\in \ov \bB_{E_\alpha}(0,\ve/c_{1}\big)\,.
\]
Together with \eqref{7a}-\eqref{7b}  we deduce that 
\[
\wh A(u)\in C^\rho\big([0,t_0],\mathcal{L}(E_1,E_0)\big)
\quad\text{with}\quad \sup_{0\le s<t\le t_0} \frac{\|\wh A(u(t))-\wh A(u(s))\|_{\mathcal{L}(E_1,E_0)}}{(t-s)^\rho}\le Lc_2\, 
\]
and
\[
-\omega_0+\delta + \wh A(u(t))\in \mathcal{H}(E_1,E_0,\kappa,\delta)\,,\quad t\in [0,t_0]\,, 
\]
for all $u^0\in\ov \bB_{E_\alpha}(0,\ve/c_{1}\big)$. Now,
\cite[II.Remark 2.1.2, II.Theorem 5.3.1]{Am95} along with \eqref{10b} yield that there is $c_3>0$
such that 
\begin{equation*} 
\begin{split}
\|u(t)-u(s)\|_{\eta} &\le c_3 \big( \|u^0\|_\alpha +\|\wh f(u)\|_{L_\infty((0,t),E_0)}\big) (t-s)^\rho\le c_3\left( \frac{\ve}{c_1} + \ve\xi\right) (t-s)^\rho
 \end{split}
\end{equation*}
for $0\leq s\leq t\leq t_0$, so that, after making  $\ve$ smaller again,  we may assume that $u\vert_{[0,t_0]}\in \mathcal{M}(t_0)$
for all  $u^0\in\ov \bB_{E_\alpha}(0,\ve/c_{1}\big).$

Set $t_1:=\sup\{t<t^+(v^0)\,;\, u\vert_{[0,t]}\in\mathcal{M}(t) \}\ge t_0$.
Then, it follows from \eqref{10b}, \eqref{b1}-\eqref{b2} and
 \cite[II.Theorem 5.3.1]{Am95} that there is a constant $c>0$  such that
\begin{equation*} 
\begin{split}
\|u(t)-u(s)\|_{\eta} &\le ce^{\nu t}\big( \|u^0\|_\alpha +\|\wh f(u)\|_{L_\infty((0,t),E_0)}\big) 
(t-s)^\rho\le ce^{\nu t}\big( \|u^0\|_\alpha + \ve\xi\big) (t-s)^\rho
\end{split}
\end{equation*}
for $0\le s\le t<t_1$  and $u^0\in\ov \bB_{E_\alpha}(0,\ve/c_{1}\big)$, where
$$
\nu:= c_0(\rho)N^{1/\rho}-\omega_0+\delta+\delta=-\omega-\delta <-\omega<0\,.
$$
In particular,
\begin{equation*} 
\begin{split}
\|u(t)\|_{ \eta} &\le \|u^0\|_{\eta} +c\big( \|u^0\|_\alpha  +  \ve\xi\big) \left(\sup_{\tau \ge 0} e^{\nu \tau}\tau^\rho\right)
\end{split}
\end{equation*}
for $0\le t<t_1$. Therefore, choosing $\xi>0$ and $\ve>0$ sufficiently small, we see that there is $\ve_0\in (0,\ve)$ such that,  
if  $u^0\in \ov\bB_{E_\alpha}(0,\ve_0)$, then
$$
\|u(t)\|_{\eta}\le \ve/c_{\eta,\beta}\,,\quad \|u(t)-u(s)\|_{ \eta} \le \frac{N}{2Lc_{\eta,\beta}} (t-s)^\rho\,,\qquad 0\le s\le t< t_1\,,
$$
with $\eta\in (\beta,\alpha)$, hence $t_1=t^+(v^0) =\infty$ in view of Theorem~\ref{T1} (v).

To sum up we have shown that, given any $ \omega\in (0,\omega_0)$, there exist  $\ve>0$ and $\ve_0\in (0,\ve)$ such that
$$
u(t)\in\ov\bB_{E_{\beta}}(0,2\ve)\subset O_{\eta} \,,\quad \|u(t)-u(s)\|_{\beta} \le \frac{N}{L} (t-s)^{\alpha-\beta}\,,\qquad t,s\ge 0\,,
$$
whenever $u^0\in \ov\bB_{E_\alpha}(0,\ve_0)$. It now follows from \eqref{7a}-\eqref{7b} and \cite[II.Lemma 5.1.3]{Am95} that
\begin{equation*}\label{88}
\| U_{\wh A(u)}(t,s)\|_{\mathcal{L}(E_\alpha)}+(t-s)^\alpha \| U_{\wh A(u)}(t,s)\|_{\mathcal{L}(E_0,E_\alpha)} 
\le c e^{\nu (t-s)}\,,\quad 0\leq s<t\,.
\end{equation*}
From this, \cite[II.Remarks 2.1.2]{Am95} and \eqref{10b} we readily obtain that
\begin{equation*}
\begin{split}
e^{\omega t}\|u(t)\|_\alpha &\le e^{\omega t}\| U_{\wh A(u)}(t,0)\|_{\mathcal{L}(E_\alpha)} \|u^0\|_\alpha 
+\int_0^t e^{\omega t}\| U_{\wh A(u)}(t,s)\|_{\mathcal{L}(E_0,E_\alpha)} \| \wh f(u(s))\|_0 \,\rd s\\
&\le c  e^{-\delta t} \|u^0\|_\alpha 
+ \xi c c_{\alpha,\beta}\int_0^t e^{\omega(t-s)} e^{\nu(t-s)}(t-s)^{-\alpha} e^{\omega s} \|u(s)\|_\alpha\,\rd s\,,
\end{split}
\end{equation*}
hence $z(t):=\max_{0\le s\le t} e^{\omega s} \|u(s)\|_\alpha$ satisfies
$$
z(t)\le c  \|u^0\|_\alpha +\xi c c_{\alpha,\beta} \int_0^\infty e^{-\delta \tau}\tau^{-\alpha}\,\rd \tau\, z(t)\,,\quad t\ge 0\,.
$$
Choosing $\xi>0$  beforehand sufficiently small, we find
$$
e^{\omega t}\|u(t)\|_\alpha \le 2 c  \|u^0\|_\alpha\,,\quad t\ge 0\,.
$$
Recalling that $v(t)=u(t)+v_*$ and $v^0=u^0+v_*$, the   proof of Theorem~\ref{T2} is complete.
\end{proof}

We now prove   the  instability result  in $E_\alpha$ stated in  Theorem \ref{T3}.
The proof  relies on the  corresponding instability result  in $E_1$ established for fully nonlinear parabolic problems in
\cite[Theorem 9.1.3]{L95}. 

\begin{proof}[Proof of Theorem \ref{T3}]
 Suppose \eqref{a1a}-\eqref{a1c} and let $v_*\in O_1$ be an equilibrium to \eqref{EE} such that \eqref{a2a}-\eqref{a2c} hold.
 Given $v^0\in O_\alpha$, let $v(\cdot;v^0)$ be the corresponding solution found in Theorem \ref{T1} and set 
 \[
 u^0:=v^0-v_*\,,\quad u:=v(\cdot;v^0)-v^*\,.
 \]
Then $u$ is a solution of the evolution problem
 \begin{align}\label{NEE}
  \dot u=\bA u+F(u)\,,\quad t>0\,,\qquad u(0)=u^0\,,
 \end{align}
where,  given  $w\in \wh O_1:=O_1-v_*$, we have set
\begin{align*}
 F(w):=-A(w+v^*)(w+v^*)+f(w+v^*)+A(v_*)w+(\p A(v_*)[w])v_*-\p f(v_*)w\,.
\end{align*}

Note that $F\in C^{2-}(\wh O_1, E_0)$ with $F(0)=0$ and $\p F(0)=0$. 
Moreover, condition \eqref{a2b}  and \cite[I.Theorem 1.3.1]{Am95} ensure that $-\bA\in\mathcal{H}(E_1,E_0)$, so that we 
are in a position to use \cite[Theorem 9.1.3]{L95} for equation \eqref{NEE}.
Hence, there exists a nontrivial backward solution
\[
z\in C((-\infty,0], E_1)\cap C^1((-\infty,0],E_0)
\]
to \eqref{NEE}, that is, $z$ satisfies
$$\dot z=\bA z+F(z)\,,\quad t\le 0\,,\qquad z(0)\not=0\,,
$$
 and,  for some $\omega\in(0,\omega_+)$, 
\[
\sup_{t\leq 0} \big(e^{-\omega t}\|z(t)\|_1\big)<\infty\,.
\]
For $ 1\leq n\in\N$ we thus find $t_n<0$ such that $\|z(t_n)\|_\alpha<1/n$.
The function $$v_n:=v_*+z(\cdot+t_n):[0,-t_n]\to O_\alpha$$ is then a  solution to \eqref{EE} with initial value  $v_n^0:=v_*+z(t_n)\in O_\alpha$
 and
\[
\|v_n(t)-v_n(s)\|_\beta\leq \|v_n(t)-v_n(s)\|_0^{1-\beta}\|v_n(t)-v_n(s)\|_1^\beta\leq C_n|t-s|^{1-\beta}\,,\quad 0\leq s\leq t\leq t_n\,.
\]
 Theorem~\ref{T1} guarantees that $-t_n<t^+(v_n^0)$ and $v_n=v(\cdot;v_n^0)\vert_{[0,t_n]}.$  
Let $U$   be a   neighborhood of $v_*$ in $O_\alpha$ with $v_*+z(0)\not\in U$. The  statement of Theorem~\ref{T3} now follows from the observation that 
$v(-t_n;v_n^0)=v_*+z(0)\not\in U$.
\end{proof}

\section{Examples}\label{Sec:4}

We  present two applications of our stability results.

\begin{exmp}[{\bf A Muskat problem without surface tension}]\label{E1}
 The horizontally  periodic motion of two immiscible fluid layers of unbounded  heights and equal viscosities (denoted by $\mu$) in a 
homogeneous porous medium with permeability 
$k$ can be described, under the assumption that  the fluid system is
close to the rest state far away from the interface separating the fluids,  by the  following equation
\begin{equation}\label{PE1}
\left\{ 
\begin{array}{rlll}
 \p_tf(t,x)\!\!\!\!\!&=&\!\!\!\!\!-\displaystyle\frac{k\Delta_\rho}{4\pi\mu}\, \p_xf(t,x) 
 \,{\rm PV}\int_{-\pi}^{\pi}\p_xf(t,x-s)\frac{(T_{[x,s]} f(t))(1+t_{[s]}^2)}{t_{[s]}^2+(T_{[x,s]} f(t))^2 }\,\rd s\\[2ex]
 &&\!\!\displaystyle-\frac{k\Delta_\rho}{4\pi\mu} \, {\rm PV} \int_{-\pi}^{\pi}\p_xf(t,x-s)
 \frac{t_{[s]}[1-(T_{[x,s]} f(t))^2]}{t_{[s]}^2+(T_{[x,s]} f(t))^2 }\,\rd s,\qquad t>0,\,\quad x\in\R\,, \\[1ex]
 f(0)\!\!\!\!\!&=&\!\!\!\!\!f^0,
\end{array}\right.
\end{equation}
see e.g. \cite{CCG11, MM17x}.
The motion is additionally assumed to be two-dimensional and  the unknown $f=f(t,x)$ in \eqref{PE1} is the function parameterizing the 
sharp interface between the fluids.
Furthermore, $g$ is  the Earth's gravity, $\rho_-$ [resp. $\rho_+$] is the density of the fluid located beneath  [resp. above]
the graph $[y=f(t,x)]$,     ${\rm PV}$ stands for 
the principle value,
and \eqref{PE1} is considered in the regime where the Rayleigh-Taylor condition
\begin{equation*}\label{RTC}
\Delta_\rho:=g(\rho_--\rho_+)>0 
\end{equation*}
holds.
 Surface tension effects have been neglected in the derivation of \eqref{PE1} and 
the  abbreviations
\[
\delta_{[x,s]}f:=f(x)-f(x-s)\,,\qquad T_{[x,s]} f=\tanh\Big(\frac{\delta_{[x,s]}f}{2}\Big)\,, \qquad t_{[s]}=\tan\Big(\frac{s}{2}\Big)\,,
\]
are used.
Though not obvious at first glance, the problem $\eqref{PE1}$ has a quasilinear structure.
Indeed, given $r\in(3/2,2)$ define
\begin{align*} 
 \Phi(f)[h](x)&:= -\frac{k\Delta_\rho}{4\pi\mu}
 \PV\int_{-\pi}^{\pi}\frac{\delta_{[x,s]}(\p_x h)}{t_{[s]}}\frac{1}{1+\big(T_{[x,s]} f/t_{[s]}\big)^2 }\,\rd s \\[1ex]
&\hspace{0.424cm}+\frac{k\Delta_\rho}{4\pi\mu} \p_xh(x)\PV\int_{-\pi}^{\pi}\Big[\frac{\p_xf(x-s)}{t_{[s]}}\frac{T_{[x,s]} f}{t_{[s]}}+\frac{1}{t_{[s]}}  \Big]
\frac{1}{1+\big(T_{[x,s]} f/t_{[s]}\big)^2 } \, \rd s\\[1ex]
 &\hspace{0.424cm}+\frac{k\Delta_\rho}{4\pi\mu} \p_xh(x) \int_{-\pi}^{\pi}\p_xf(x-s)\frac{T_{[x,s]} f}{1+\big(T_{[x,s]} f/t_{[s]}\big)^2 }\,\rd s \\[1ex]
 &\hspace{0.424cm}- \frac{k\Delta_\rho}{4\pi\mu} 
 \int_{-\pi}^{\pi}\p_xh(x-s)\frac{\big(T_{[x,s]} f/t_{[s]}\big)T_{[x,s]} f}{1+\big(T_{[x,s]} f/t_{[s]}\big)^2 }\,\rd s
\end{align*}
for $f\in H^r(\s)$, $h\in H^2(\s)$, and $x\in\R$. Then 
\eqref{PE1} can be recast in the form
\begin{equation*} 
\dot f+ \Phi(f)[f] =0\,, \quad t>0\,,\qquad f(0)=f^0,
\end{equation*}
where
\begin{align*} \label{PPE1}
  \Phi\in {\rm C}^{\omega}(H^r(\s),\mathcal{H}(H^2(\s), H^1(\s)))\qquad\text{for  $r\in(3/2,2)$},
 \end{align*}
 see \cite{MM17x}.
Choosing $(\cdot,\cdot)_\theta$  in Theorem \ref{T1} to be the complex interpolation functor,
it follows that the problem \eqref{PE1}
is well-posed in $H^r(\s)$ for each $r\in(3/2,2)$, cf. \cite[Theorem 1.1]{MM17x}.

According to \cite[Remark 3.4]{M18x}, the equilibria to \eqref{PE1} lying
in $H^2(\s)$ are  the constant functions. Moreover, the flow \eqref{PE1} preserves the integral mean of the initial data.
 Hence, when studying the stability properties of the zero  solution, it is natural to consider perturbations with zero integral mean.
 Letting 
 \[
H^r_0(\s):=\Big\{f\in H^r(\s)\,:\, \int_0^{2\pi}f\,\rd x=0\Big\}\,, \qquad r\geq0\,,
\]
it is shown in  \cite{MM17x} that, in fact,
\begin{align*}  
  \Phi\in {\rm C}^{\omega}(H^r_0(\s),\mathcal{H}(H^2_0(\s), H^1_0(\s)))\qquad\text{ for  $r\in(3/2,2)$}.
 \end{align*}
 Since the spectrum of the linearization $-\p\Phi(0)\in\mathcal{L}(H^2_0(\s), H^1_0(\s))$ is given by 
 $$\sigma(-\p\Phi(0))=\{-k\Delta_\rho m/(2\mu)\,:\, m\in\N\setminus\{0\}\},$$
 see \cite{MM17x}, we are in a position to apply Theorem \ref{T2} and  conclude  that the zero solution is exponentially stable in the natural
 phase space 
 $H^r_0(\s)$.

 \begin{thm}[Exponential stability]\label{MT2}
  Let  $\Delta_\rho>0$ and $r\in(3/2,2)$. Then, given $\omega\in(0,k\Delta_\rho/2\mu)$, there exist constants $\ve_0>0$ and $M>0$ such that,
  for each $f^0\in H^r_0(\s)$  with $\|f^0\|_{H^{r}(\s)}\leq \ve_0$, the solution  $f(\cdot;f^0)$ to \eqref{PE1} exists globally in time and
  \begin{align*}
   \|f(t;f^0)\|_{H^r(\s)}\leq Me^{-\omega t}\|f_0\|_{H^{r}(\s)}\,,\quad t\geq 0\,.
  \end{align*} 
 \end{thm}
 
 The exponential stability of the zero solution under perturbations in $H^2_0(\s)$ has been established only recently in
   \cite[Theorem 1.3]{MM17x} by using the fully nonlinear principle of linearized stability from \cite{L95} and the abstract 
   formulation presented above. 
   Thus, Theorem~\ref{MT2} improves this result to stability in the optimal phase space $H_0^r(\s)$ with $r\in (3/2,2)$. Let us also mention that,
   by means of energy  techniques, the authors of \cite{BCS16} have previously derived 
   decay estimates  with respect to  $H^r(\s)$-norms, $r\in[0,2)$,  for solutions
   corresponding to small  initial data  in $H^2_0(\s)$.
\end{exmp}

\begin{exmp}[{\bf The Muskat problem with  surface tension}]\label{E2}
 The physical scenario   in this example is similar to  the one  in Example \ref{E1}, 
 but now we consider the   general case when the viscosities of the fluids  are not necessarily equal, that is,
 $\mu_--\mu_+\in\R.$
 In addition, the fluid system  moves with constant velocity $(0,V)$, with $V\in\R$, in the vertical plane, the interface between
 the fluids is  the graph 
 $[y=f(t,x)+tV]$, and  the fluid  velocities are set to be
asymptotically equal  to $(0,V)$ far away from the interface.
Moreover, surface tension effects are taken into account at the interface between the   fluids, and  we denote by
$\sigma>0$ the surface tension coefficient and by $\kappa(f)$ the curvature of the moving interface.
The mathematical model consists of the following set of equations
\begin{equation}\label{PE2}
\left\{ 
\begin{array}{rllll}
 \displaystyle\p_tf(t,x)\!\!\!\!\!&=&\!\!\!\!\!  \displaystyle\frac{1}{4\pi} 
 \PV\int_{-\pi}^{\pi}\frac{\p_xf(t,x) (1+t_{[s]}^2)(T_{[x,s]}f(t))+t_{[s]}[1-(T_{[x,s]}f(t))^2] }{t_{[s]}^2+(T_{[x,s]}f(t))^2 }
 \ov\omega(t,x-s)\, \rd s\,,\\[2ex]
\ov\omega(t,x)\!\!\!\!\!&=&\!\!\!\!\! \displaystyle\frac{2k}{\mu_-+\mu_+}\p_x(\sigma\kappa(f(t))-\Theta f(t))(x)\\[2ex]
&&\!\!\!\!-\displaystyle \frac{a_\mu}{2\pi } \PV\int_{-\pi}^{\pi}\frac{\p_xf(t,x)t_{[s]}[1-(T_{[x,s]}f(t))^2]
- (1+t_{[s]}^2)T_{[x,s]}f(t)}{t_{[s]}^2+(T_{[x,s]}f(t))^2 }\ov\omega(t,x-s)\, \rd s\,,
\end{array}\right.
\end{equation}
 for $t\geq 0 $ and $ x\in\R,$ which is supplemented by the initial condition
 \begin{equation}\label{ICP}
 f(0)=f^0.
\end{equation}
Furthermore, $\Theta$ and $a_\mu$  are the constants 
\begin{equation*}
\Theta:=g(\rho_--\rho_+)+\frac{\mu_--\mu_+}{k}V,\qquad a_\mu:=\frac{\mu_--\mu_+}{\mu_-+\mu_+}.
\end{equation*}
The function $\ov\omega$ is also unknown, but can be determined by $f$. Indeed, given $f\in H^2(\s)$ and $h\in H^3(\s)$, the 
equation
\[
(1+a_\mu\bA(f))[\ov\omega]=\frac{h'''}{(1+f'^2)^{3/2}}-3\frac{f'f''h''}{(1+f'^2)^{3/2}}-\Theta h'\,,
\]
where
\begin{align} \label{DFDD}
 \bA(f)[\ov\omega](x):=\frac{1}{2\pi }\PV\int_{-\pi}^{\pi}\frac{\p_xf(x)t_{[s]}[1-(T_{[x,s]}f)^2]- (1+t_{[s]}^2)T_{[x,s]}f}{t_{[s]}^2
 +(T_{[x,s]}f)^2 }\ov\omega(x-s)\, \rd s\,,
\end{align}
has a unique solution $\ov\omega=:\ov\omega(f)[h]\in L_{2,0}(\s),$ cf. \cite[Proposition 4.1]{M18x}.
Letting 
 \begin{align*} 
  \bB(f)[\ov\omega](x):= \frac{1}{2\pi} \PV\int_{-\pi}^{\pi}\frac{\p_xf(x) (1+t_{[s]}^2)(T_{[x,s]}f)+t_{[s]}[1-(T_{[x,s]}f)^2] }
  {t_{[s]}^2+(T_{[x,s]}f)^2 }\ov\omega(x-s)\, \rd s,
 \end{align*}
 the evolution problem \eqref{PE2}-\eqref{ICP} can be formulated as
 \begin{equation*} 
\dot f+\Phi(f)[f] =0\,, \quad t>0\,,\qquad f(0)=f^0,
\end{equation*}
where
$$ \Phi(f)[h]:=-\frac{\sigma k}{\mu_-+\mu_+}\bB(f)[\ov\omega(f)[h]]\,, \qquad f\in H^2(\s)\,,\quad h\in H^3(\s)\,,$$
satisfies
\begin{align*}  
  \Phi\in {\rm C}^{\omega}(H^2(\s),\mathcal{H}(H^3(\s), L_{2}(\s)))\cap {\rm C}^{\omega}(H^2_0(\s),\mathcal{H}(H^3_0(\s), L_{2,0}(\s)))\,,
 \end{align*}
 cf. \cite[Section 4]{M18x}. 
 Theorem \ref{T1} yields then the well-posedness of the problem \eqref{PE2}-\eqref{ICP} in the phase space $H^r(\s)$ for each $r\in(2,3)$.
 
 Concerning the stability issue, we point out that constants are again equilibria and the Theorems \ref{T2}-\ref{T3} can be used to study the 
 stability properties of the zero solution in the phase space $H^r(\s)$ and with respect to perturbations with zero integral mean. 
 Indeed, the linearized operator $-\p\Phi(0)\in\kL(H^3_0(\s), L_{2,0}(\s))$ has spectrum given by
 \[\sigma(-\p\Phi(0))=\Big\{-\frac{\sigma k}{\mu_-+\mu_+} \Big(m^3+\frac{\Theta}{\sigma}m\Big)\,:\,m\in\N\setminus\{0\}\Big\}.\] 
 Together with Theorem \ref{T2}-\ref{T3}  it follows  that the sign of $\Theta+\sigma$ determines  the stability and instability of
 the zero solution in the phase space $H^r(\s)$. 

\begin{thm}[{\cite[Theorem 1.3]{M18x}}]
 Let $r\in(2,3)$. 
 \begin{itemize}
 \item[(a)]  If $\Theta+\sigma>0$, then the zero solution is stable in $H^r(\s)$. More precisely,
given  $$\omega\in(0,k(\sigma+\Theta)/(\mu_-+\mu_+))\,,$$  there are constants $\ve_0>0$ and $M>0$,
such that, if $f^0\in  H^r_0(\s)$ satisfies 
 \mbox{$\|f^0\|_{H^{r}(\s)}\leq\ve_0$}, 
   then the solution   $f(\cdot;f^0)$ to \eqref{PE2}-\eqref{ICP} exists globally  and
  \begin{align*}
   \|f(t;f^0)\|_{H^r(\s) } \leq Me^{-\omega t}\|f^0\|_{H^{r}(\s) }\,,\qquad \text{$t\geq 0$\,.}
  \end{align*}
    \item[(b)]   If $\Theta+\sigma<0$, then the zero solution is unstable in $H^r(\s)$. More precisely, there exist  
   $R>0$ and a sequence $(f_n^0)\subset  H^{r}_0(\s)$ of initial data  such that \\[-2ex]
\begin{itemize}
\item[$\bullet$] $f_n^0\to0 $ in $ H^{r}(\s)$,\\[-2ex]
\item[$\bullet$] there exists $t_n\in(0,t^+(f_n^0))$ with $\|f(t_n;f^0_n)\|_{H^r(\s)}=R$.\\[-2ex]
\end{itemize}
\end{itemize}
\end{thm}

We point out that  here, but also  in Example \ref{PE1}, the other constant equilibria have the same stability properties as the zero solution.
Moreover, in the context of the Muskat problem \eqref{PE2} there may exist also other, finger-shaped equilibria (see \cite[Section 6]{M18x}
for a complete classification of the equilibria). Theorem~\ref{T3} can be used to prove that  small (with respect to the $H^r(\s)$-norm)
finger-shaped equilibria are unstable, cf.  \cite[Theorem 1.5 (iii)]{M18x}.
 \end{exmp}
\bibliographystyle{siam}
\bibliography{Literature}
\end{document}